\documentclass[10pt,a4paper]{amsart}
\usepackage[T1]{fontenc}
\usepackage[utf8]{inputenc}
\usepackage{microtype} 	
\usepackage{enumitem}
\setlist[enumerate]{leftmargin=*}
\usepackage[dvipsnames]{xcolor}
\usepackage[backref]{hyperref}
\hypersetup{colorlinks=true,linkcolor=Maroon,citecolor=OliveGreen,urlcolor=OliveGreen}
\usepackage{amsmath,amsthm,amssymb}
\usepackage{mathtools}
\usepackage{autobreak} % for breaking long equations
\allowdisplaybreaks % break long equations across pages
 \usepackage{bm}% load after all math to give access to bold math
\usepackage{bbold}
\numberwithin{equation}{section}
\theoremstyle{definition}

\newtheorem{example}{Example}
\theoremstyle{remark}
\newtheorem{remark}{Remark}
\theoremstyle{plain}
\newtheorem{proposition}{Proposition}
\newtheorem{lemma}{Lemma}
\newtheorem{corollary}{Corollary}
\newtheorem{theorem}{Theorem}
\def\CC{\mathbb{C}}

\def\ZZ{\mathbb{Z}}

\def\FF{\mathsf{F}}
\def\wiggle{\mathsf{Wiggle}}
\def\unipotents{\mathsf{Uni}}

\DeclareMathOperator{\GL}{GL}
\DeclareMathOperator{\Mat}{Mat}
\DeclareMathOperator{\Diag}{Diag}
\DeclareMathOperator{\Torus}{Torus}
\DeclareMathOperator{\SL}{SL}
\DeclareMathOperator{\PSL}{PSL}

\DeclareMathOperator{\tr}{tr}
\DeclareMathOperator{\image}{im}
\DeclareMathOperator{\res}{Res}
\usepackage{todonotes}

\title{On surjectivity of word maps on $\PSL_2$}

\author{Urban Jezernik}
\address{Urban Jezernik, 
Alfréd Rényi Institute of Mathematics, 
Hungarian Academy of Sciences, 
Reáltanoda utca 13-15, 
H-1053, Budapest, 
Hungary}
\email{jezernik.urban@renyi.hu}

\author{Jonatan Sánchez}
\address{Jonatan S\'anchez,
Department of Applied Mathematics (DMATIC), ETSI Ingenieros Informáticos,
Universidad Polit\'ecnica de Madrid,
Campus de Montegancedo, Avenida de Montepríncipe,
28660, Boadilla del Monte,
Spain
}
\email{jonatan.sanchez@upm.es}

\thanks{The first author has received funding from 
the European Union’s Horizon 2020 research and innovation programme
under the Marie Sklodowska-Curie grant agreement No. 748129,
as well as funding from the European Research Council (ERC) 
under the European Union’s Horizon 2020 research and innovation programme 
(grant agreement No. 741420).
The second author is partially supported by Universidad Polit\'ecnica de Madrid (UPM)}

\date{\today}
\begin{document}
\baselineskip=15pt
%%%%%%%%%%%%%%%%%%%%%%%%%%%%%%%%%%%%%%%%%%%%%%%%%%%%%%%%%%%%%%%%%%%%%%%%%%%%%%%%%%%%%
%%%%%%%%%%%%%%%%%%%%%%%%%%%%%%%%%%%%%%%%%%%%%%%%%%%%%%%%%%%%%%%%%%%%%%%%%%%%%%%%%%%%%

\begin{abstract}
Let $w = [[x^k, y^l], [x^m, y^n]]$ be a non-trivial
double commutator word. We show that $w$ is surjective on
$\PSL_2(K)$, where $K$ is an algebraically closed field of
characteristic $0$.
\end{abstract}

\maketitle

% \tableofcontents
%%%%%%%%%%%%%%%%%%%%%%%%%%%%%%%%%%%%%%%%%%%%%%%%%%%%%%%%%%%%%%%%%%%%%%%%%%%%%%%%%%%%%
%%%%%%%%%%%%%%%%%%%%%%%%%%%%%%%%%%%%%%%%%%%%%%%%%%%%%%%%%%%%%%%%%%%%%%%%%%%%%%%%%%%%%
\section{Introduction}
%%%%%%%%%%%%%%%%%%%%%%%%%%%%%%%%%%%%%%%%%%%%%%%%%%%%%%%%%%%%%%%%%%%%%%%%%%%%%%%%%%%%%
%%%%%%%%%%%%%%%%%%%%%%%%%%%%%%%%%%%%%%%%%%%%%%%%%%%%%%%%%%%%%%%%%%%%%%%%%%%%%%%%%%%%%

%%%%%%%%%%%%%%%%%%%%%%%%%%%%%%%%%%%%%%%%%%%%%%%%%%%%%%%%%%%%%%%%%%%%%%%%%%%%%%%%%%%%%
\subsection{Words, word maps and their surjectivity}
%%%%%%%%%%%%%%%%%%%%%%%%%%%%%%%%%%%%%%%%%%%%%%%%%%%%%%%%%%%%%%%%%%%%%%%%%%%%%%%%%%%%%
A \emph{word in two variables} $w$ is an element of the free group 
$\FF_2 = \langle x, y \rangle$. Given a group $G$, the word $w$ induces a
\emph{word map} $\tilde w$ on $G$ by evalution,
\[
\tilde w \colon G \times G \to G, \quad (g,h) \mapsto w(g,h).
\]

When the underlying group $G$ is a connected semisimple algebraic group, 
say $\SL_n(K)$ for an algebraically closed field $K$, 
every non-trivial word map is dominant by a theorem of Borel \cite{borel1983free}.

For certain words, one can even prove surjectivity and possibly further properties
of the word map, say flatness.
All of these can then be used to descend to the case of finite simple groups, 
say $\PSL_n(\mathbb{F}_q)$,
and deduce surjectivity or even uniform distribution there.
See \cite{larsen2019probabilistic,larsen2019most} 
for a recent application of this technique.

Despite being dominant, not all word maps are surjective on linear algebraic groups, 
for example taking powers on $\SL_2(\CC)$ is not always surjective. 
The situation is different for adjoint groups -- 
the surjectivity problem asking whether or not all word maps are surjective
is still open for the case of two variable word maps on $\PSL_2(K)$,
where $K$ is algebraically closed of characteristic $0$. It is this problem
that we address in the present paper.

It was shown by Bandman and Zarhin \cite{bandman2016surjectivity},
and later reproved by Gordeev, Kunyavski{\u\i} and Plotkin \cite{gordeev2018word},
that the surjectivity problem for $\PSL_2(K)$ has a positive solution for words $w$
{\em not} belonging to the second derived subgroup $\FF_2^{(2)}$ of $\FF_2$.
Their results furthermore imply that for any word $w \in \FF_2$,
the image $\image \tilde w$ contains all semisimple
elements of $\PSL_2(K)$. As $\image \tilde w$ is closed for conjugation,
the surjectivity problem is then reduced to finding a single 
non-trivial unipotent in $\image \tilde w$ for $w \in \FF_2^{(2)}$.
This has been shown to hold for a handful of concrete words, and recently Gnutov and Gordeev \cite{gnutov2020recursive} showed that each such example $w \in \FF_2^{(2)}$ can be used to produce a sequence of words $\{ w_i \}_{i \geq 2}$ with $w_i \in \FF_2^{(i)}$ so that $\tilde w_i$ is also surjective. Here $\FF_2^{(i)}$ is the $i$-th term of the derived series of $\FF_2$.

% Recall the sequence of derived subgroups $\mathbb{F}_2$, $\mathbb{F}_2^{(1)}=[\mathbb{F}_2,\mathbb{F}_2]$,
% $\mathbb{F}_2^{(2)}=[\mathbb{F}_2^{(1)},\mathbb{F}_2^{(1)}]$, etc. where 
% $[G,G]=\{ghg^{-1}h^{-1}\,|\,g,h\in G\}$ denotes the commutator subgroup of the group $G$.
% %$[\mathbb{F},\mathbb{F}]=\{ghg^{-1}h^{-1}\,|\,g,h\in
% % \mathbb{F}\}$ denotes the commutator subgroup of the group $\mathbb{F}$.

%%%%%%%%%%%%%%%%%%%%%%%%%%%%%%%%%%%%%%%%%%%%%%%%%%%%%%%%%%%%%%%%%%%%%%%%%%%%%%%%%%%%%
\subsection{Main contribution and strategy of proof}
%%%%%%%%%%%%%%%%%%%%%%%%%%%%%%%%%%%%%%%%%%%%%%%%%%%%%%%%%%%%%%%%%%%%%%%%%%%%%%%%%%%%%

The second derived subgroup $\FF_2^{(2)}$ is generated as a normal subgroup of
$\FF_2$ by the {\em double commutators}
\[
[[x^k, y^l], [x^m, y^n]]
\quad
\text{with }
k,l,m,n \in \ZZ.
\]
Here $[x,y]$ stands for $x^{-1} y^{-1} x y$. 
We show that the surjectivity problem for $\PSL_2(K)$
has a positive solution for these generating words.

\begin{theorem} \label{theorem:double_commutators_are_surjective}
Every non-trivial double commutator induces a surjective word map on $\PSL_2(K)$,
where $K$ is an algebraically closed field of characteristic $0$.
\end{theorem}

The method of the proof works for general words and is restricted to double commutators
only in its final stage. Our strategy is inspired by that of \cite{gordeev2018word},
where the authors show how the existence of unipotents in $\image \tilde w$
is related to the geometric structure of the representation variety
of the $1$-relator group $\FF_2 / \langle w \rangle$.
The required computations to extract this geometric structure 
are made easier if one is able to pass to the quotient variety by the action of $\SL_2(\CC)$.
In our approach, we work directly inside the group $\SL_2(K)$.
Note that it suffices to show the existence of a non-trivial unipotent
in $\image \tilde w$ for this group.
We focus on the Zariski open subset of semisimple pairs in $G$
and consider its image in the quotient by $\SL_2(K)$.
We call this procedure {\em wiggling} (see Section \ref{sec:wiggling}).
Our ultimate aim is to show that one can find diagonal matrices $x,y \in \SL_2(K)$
for which the image of the wiggle map contains a non-trivial unipotent.

\begin{example}
Let 
$x = \begin{psmallmatrix} \lambda & 0 \\ 0 & \lambda^{-1} \end{psmallmatrix}$
and
$y = \begin{psmallmatrix} \mu & 0 \\ 0 & \mu^{-1} \end{psmallmatrix}$
be a pair of diagonal matrices in $\SL_2(K)$. 
The wiggle map associated to the ordinary commutator word is
\[
\wiggle^{x,y} \colon \SL_2(K) \to \SL_2(K), \quad g \mapsto [x, y^g].
\]
One directly computes that, for
$g = \begin{psmallmatrix} a & b \\ c & d \end{psmallmatrix}$,
we have
\[
\tr \wiggle^{x,y}(g)
=
\left(\lambda - \lambda^{-1}\right)^2 \left(\mu - \mu^{-1}\right)^2 
bc \left( bc + 1 \right)
+ 2.
\]
As long as $\lambda, \mu \neq \pm 1$, the trace is equal to $2$ if and only if
$bc \in \{ 0, -1 \}$.
On the other hand,
the $(1,2)$-entry of the image of $g$ under the wiggle map is
\[
\wiggle^{x,y}(g)^2_1 =
\left(\lambda - \lambda^{-1}\right)^2 \left(\mu - \mu^{-1}\right)^2
\left( 
1 + \left( 1 - \mu^2 \right) bc
\right)
bd.
\]
As long as $\lambda, \mu \neq \pm 1$, this entry is equal to $0$ if and only if
$bd = 0$ or $bc = - 1/(1 - \mu^2)$.
Taking $\lambda = \mu = 2$ and 
$g = \begin{psmallmatrix} 0 & 1 \\ -1 & 1 \end{psmallmatrix}$,
we see that the image of the wiggle map contains a non-identity
element of trace $2$, which must be a non-trivial unipotent.
\end{example}

We first show that the wiggle map associated to any word 
can be expressed in a normal form in terms of
two matrix polynomials in the variable $bc$ as in the example above.
We also explain the meaning of the coefficients of these polynomials.
We then give an effective way of computing these polynomials.
Furthermore, we show that the wiggle map has a particularly symmetric form.
This can be exploited (see Section \ref{sec:finding_unipotents}) 
to reduce the problem of finding non-trivial unipotents
to the problem of finding a particular root of a single polynomial,
similar to the one in the example above.
This polynomial is related to the trace polynomial 
as developed in general 
by Fricke \cite{fricke1897vorlesungen}.
For the purposes of this paper, however, additional
information regarding the structure of this polynomial
in needed and we achieve this with the matrix form of the 
wiggle map. 
We conclude by applying all of the above to the case of 
double commutators, 
where we are able to explicitly compute the relevant polynomial.

Some of the symbolic calculations were verified using {\sf Wolfram Mathematica}, the relevant notebook is available on the website of the first author.\footnote{\tiny\url{https://drive.google.com/file/d/15AKylXoKmwW4hsnSPhv5GvUulRSkFZPQ/view?usp=sharing}}

%%%%%%%%%%%%%%%%%%%%%%%%%%%%%%%%%%%%%%%%%%%%%%%%%%%%%%%%%%%%%%%%%%%%%%%%%%%%%%%%%%%%%
\subsection{Notation}
%%%%%%%%%%%%%%%%%%%%%%%%%%%%%%%%%%%%%%%%%%%%%%%%%%%%%%%%%%%%%%%%%%%%%%%%%%%%%%%%%%%%%

We will write $\Mat_2(K)$ for the matrix algebra and
$\Diag_2(K)$ for the set of all diagonal matrices in $\Mat_2(K)$. Then
$\Torus_2(K) = \SL_2(K) \cap \Diag_2(K)$ is the standard maximal torus in $\SL_2(K)$.
We will use the same notation for more general matrix rings $\Mat_2(R)$ over a ring
$R$. For a matrix $A \in \Mat_2(R)$, we will let $A^j_i$ denote the element of $A$
at position $(i,j)$.

%%%%%%%%%%%%%%%%%%%%%%%%%%%%%%%%%%%%%%%%%%%%%%%%%%%%%%%%%%%%%%%%%%%%%%%%%%%%%%%%%%%%%
%%%%%%%%%%%%%%%%%%%%%%%%%%%%%%%%%%%%%%%%%%%%%%%%%%%%%%%%%%%%%%%%%%%%%%%%%%%%%%%%%%%%%
\section{Conjugating semisimple elements}
%%%%%%%%%%%%%%%%%%%%%%%%%%%%%%%%%%%%%%%%%%%%%%%%%%%%%%%%%%%%%%%%%%%%%%%%%%%%%%%%%%%%%
%%%%%%%%%%%%%%%%%%%%%%%%%%%%%%%%%%%%%%%%%%%%%%%%%%%%%%%%%%%%%%%%%%%%%%%%%%%%%%%%%%%%%

%%%%%%%%%%%%%%%%%%%%%%%%%%%%%%%%%%%%%%%%%%%%%%%%%%%%%%%%%%%%%%%%%%%%%%%%%%%%%%%%%%%%%
\subsection{Conjugation map}
%%%%%%%%%%%%%%%%%%%%%%%%%%%%%%%%%%%%%%%%%%%%%%%%%%%%%%%%%%%%%%%%%%%%%%%%%%%%%%%%%%%%%

Our objective will be to evaluate the word map $\tilde w$ at a pair of semisimple 
elements in a matrix group $G \subseteq \GL_2(K)$.  
To achieve this, we will transform occurrences of diagonalizable 
matrices with their diagonal forms under conjugation by some element $g \in G$. 
In order to control the situation, we therefore study the linear map
\[
C_g \colon \Mat_2(K) \to \Mat_2(K), \quad x \mapsto x^g - x.
\]
Here and throughout, we denote $x^g = g^{-1} x g$.

%%%%%%%%%%%%%%%%%%%%%%%%%%%%%%%%%%%%%%%%%%%%%%%%%%%%%%%%%%%%%%%%%%%%%%%%%%%%%%%%%%%%%
\subsection{Conjugation map on diagonal matrices}
%%%%%%%%%%%%%%%%%%%%%%%%%%%%%%%%%%%%%%%%%%%%%%%%%%%%%%%%%%%%%%%%%%%%%%%%%%%%%%%%%%%%%

Set $\xi_g = C_g(
\left(\begin{smallmatrix} 1 & 0 \\ 0 & 0\end{smallmatrix}\right)
)$.
The image of a diagonal matrix 
$\left(\begin{smallmatrix} \lambda & 0 \\ 0 & \mu \end{smallmatrix}\right)$
under $C_g$ can be expressed in terms of the matrix $\xi_g$ as follows:
\begin{equation}\label{eq:Cgx}
C_g(\left(\begin{smallmatrix} \lambda & 0 \\ 0 & \mu \end{smallmatrix}\right))
=
C_g(\left(\begin{smallmatrix} \lambda & 0 \\ 0 & 0 \end{smallmatrix}\right))
+ 
C_g(\left(\begin{smallmatrix} 0 & 0 \\ 0 & \mu \end{smallmatrix}\right)
-
\left(\begin{smallmatrix} \mu & 0 \\ 0 & \mu \end{smallmatrix}\right))
=
(\lambda - \mu) \xi_g.
\end{equation}
This means that a conjugate of a diagonal matrix can be expressed as
\begin{equation}
\left(\begin{smallmatrix} \lambda & 0 \\ 0 & \mu \end{smallmatrix}\right)^g
=
\left(\begin{smallmatrix} \lambda & 0 \\ 0 & \mu \end{smallmatrix}\right)
+
(\lambda - \mu) \xi_g.
\end{equation}
We will write this more compactly; for a diagonal matrix 
$x = \left(\begin{smallmatrix} \lambda & 0 \\ 0 & \mu \end{smallmatrix}\right)$,
let $\bar{x} = \lambda - \mu$. Then
\begin{equation} \label{eq:conj}
x^g = x + \bar{x} \xi_g.
\end{equation}

% \begin{remark}
%     If we want to avoid singular matrices, we can take $\xi_g=C_g(x)$ where
%     $x=\diag(\phi,\phi^{-1})$ for $\phi=(1+\sqrt{5})/2$. Every equality and
%     property will hold.
% \end{remark}

%%%%%%%%%%%%%%%%%%%%%%%%%%%%%%%%%%%%%%%%%%%%%%%%%%%%%%%%%%%%%%%%%%%%%%%%%%%%%%%%%%%%%
\subsection{Properties of the conjugation map}
%%%%%%%%%%%%%%%%%%%%%%%%%%%%%%%%%%%%%%%%%%%%%%%%%%%%%%%%%%%%%%%%%%%%%%%%%%%%%%%%%%%%%

Let us collect some properties of the map $C_g$ and of the matrix $\xi_g$.
The identity matrix of $\GL_2(K)$ will be denoted by $\mathbb{1}$ throughout.

% %In this point, it will be useful to find out some properties on $C_g$ and
% %$\xi_g$. We summarize these in the following lemma:
% We give some properties of $C_g$ and $\xi_g$:

\begin{lemma} \label{lemma:xiproperties}
Let $x, y \in \Mat_2(K)$ and $g, h \in \GL_2(K)$.
\begin{enumerate}
    \item $\tr C_g(x) = 0$
    \item $C_g(x)^2 = - \det C_g(x) \cdot \mathbb{1}$ 
    \item $C_{g^h}( x^h ) = C_g( x )^h$
    \item $C_g( xy ) - C_g(x) C_g(y) = x C_g(y) + C_g(x) y$
\end{enumerate}
\end{lemma}

\begin{proof}
    (1) $\tr C_g(x) = \tr x^g - \tr x = 0$.
    (2) Immediate from the Cayley-Hamilton theorem.
    (3) $C_{g^h}(x^h) = x^{h g^h} - x^h = x^{gh} - x^h = C_g(x)^h$.
    (4) Expand
    \[
        C_g(xy) - C_g(x) C_g(y) 
        = (xy)^g - xy - (x^g - x)(y^g - y) 
        = x^g y + x y^g - 2 xy,
    \]
    and notice that the latter is precisely $C_g(x)y + xC_g(y)$.
\end{proof}

It follows from the lemma that $\tr \xi_g = 0$ and $\xi_g^2 = - \det \xi_g \cdot \mathbb{1}$.
We now collect some more properties of the matrix $\xi_g$ that will be useful in the rest of the paper.

\begin{lemma}\label{lemma:xiproperties2}
Let $x \in \Diag_2(K)$.
\begin{enumerate}
    \item $\tr \left( x\xi_g \right) = \bar{x} \det \xi_g$
    % \item Set $\sigma = \left( \begin{smallmatrix} 0 & 1 \\ 1 & 0 \end{smallmatrix} \right)$.
    % Then
    %         \[
    %             \xi_gx = x^{\sigma} \xi_g + \bar{x} \det\xi_g \cdot \mathbb{1}.
    %         \]
    \item If $x \in \Torus_2(K)$, then $\xi_g x = x^{-1} \xi_g + \bar{x} \det \xi_g \cdot \mathbb{1}$.
\end{enumerate}
\end{lemma}

\begin{proof}
(1) When $\bar{x} = 0$, we have $x = \pm \mathbb{1}$ and the claim holds.
Assume now that $\bar{x} \neq 0$. 
Use Lemma \ref{lemma:xiproperties} (4) with $y = x$ to obtain
$C_g(x)^2 = - xC_g(x) - C_g(x)x + C_g(x^2)$. It follows from 
Lemma \ref{lemma:xiproperties} (2) together with \eqref{eq:conj} that
$C_g(x)^2 = - \bar{x}^2 \det \xi_g \cdot \mathbb{1}$.
We now apply the trace map to obtain
\[
- 2 \bar{x}^2 \det \xi_g = \tr C_g(x)^2 = 
- \tr x C_g(x) - \tr C(g) x + \tr C_g(x^2) = - 2 \tr x C_g(x).
\]
Note that $x C_g(x) = \bar{x} x \xi_g$ by \eqref{eq:conj}. This implies 
$\tr x \xi_g = \bar{x} \det \xi_g$.

(2) Use Lemma \ref{lemma:xiproperties} (4) with $y = x^{-1}$ to obtain
$C_g(x) C_g(x^{-1}) = - x C_g(x^{-1}) - C_g(x) x^{-1}$.
It follows from \eqref{eq:conj} that 
$C_g(x) = \bar{x} \xi_g$ and $C_g(x^{-1}) = - \bar{x} \xi_g$. We therefore have
that $- \bar{x}^2 \xi_g^2 = \bar{x} x \xi_g - \bar{x} \xi_g x^{-1}$. It follows from
Lemma \ref{lemma:xiproperties} (2) that 
$\xi_g x^{-1} = x \xi_g - \bar{x} \det \xi_g \cdot \mathbb{1}$.
The claim follows after replacing $x$ by $x^{-1}$.
\end{proof}

\begin{lemma} \label{lemma:xiproperties_formofxi}
Set $t = \det \xi_g$. Then 
$\xi_g = \left( \begin{smallmatrix} t & p \\ q & - t \end{smallmatrix} \right)$ 
for some $p,q \in K$ satisfying $p  q = - t (t+1)$.
\end{lemma}

\begin{proof}
Use Lemma \ref{lemma:xiproperties2} (1) with 
$x = \left( \begin{smallmatrix} 1 & 0 \\ 0 & 0 \end{smallmatrix} \right)$
to conclude that $(\xi_g)_1^1 = t$. The same argument with
$x = \left( \begin{smallmatrix} 0 & 0 \\ 0 & 1 \end{smallmatrix} \right)$
gives $(\xi_g)_2^2 = - t$. 
The equality $p q = - t (t+1)$ is nothing but $\det \xi_g = t$.
\end{proof}

\begin{remark} \label{remark:xig_explicit}
Writing 
$g = \left( \begin{smallmatrix} a & b \\ c & d \end{smallmatrix} \right) \in \SL_2(K)$,
we have
$\xi_g = \left( \begin{smallmatrix} bc & bd \\ - ac & - bc \end{smallmatrix} \right)$
and $\det \xi_g = bc$.
\end{remark}

\begin{example}
We show how the properties of $\xi_g$ determined in this section can be used to quickly show that the commutator word $w = [x^a, y^b]$ for $a,b$ non-zero integers contains non-trivial unipotents in its image.
To this end, take a semisimple element $y \in \Torus_2(K)$ with $y^b$ non-trivial, and let $x \in \SL_2(K)$ be a non-diagonal upper-triangular matrix.
Thus $\xi_{x^a} = \left( \begin{smallmatrix} 0 & \star \\ 0 & 0 \end{smallmatrix} \right)$ with $\star \neq 0$ by Remark \ref{remark:xig_explicit}.
We have
\[
\tilde w (x,y) =
(y^{-b})^{x^a} y^b = 
(y^{-b} + \overline{y^{-b}} \xi_{x^a}) y^b =
\mathbb{1} + \overline{y^{-b}} \xi_{x^a} y^b,
\]
hence $\tilde w(x,y)$ is a non-trivial unipotent.
\end{example}
\section{Wiggling the word map on semisimple pairs}
\label{sec:wiggling}
%%%%%%%%%%%%%%%%%%%%%%%%%%%%%%%%%%%%%%%%%%%%%%%%%%%%%%%%%%%%%%%%%%%%%%%%%%%%%%%%%%%%%
%%%%%%%%%%%%%%%%%%%%%%%%%%%%%%%%%%%%%%%%%%%%%%%%%%%%%%%%%%%%%%%%%%%%%%%%%%%%%%%%%%%%%

In this section, we deal with finding a description of the values of the word map
$\tilde w$ on pairs of semisimple elements in terms of certain polynomials.

%%%%%%%%%%%%%%%%%%%%%%%%%%%%%%%%%%%%%%%%%%%%%%%%%%%%%%%%%%%%%%%%%%%%%%%%%%%%%%%%%%%%%
\subsection{Reducing the word}
%%%%%%%%%%%%%%%%%%%%%%%%%%%%%%%%%%%%%%%%%%%%%%%%%%%%%%%%%%%%%%%%%%%%%%%%%%%%%%%%%%%%%

The word $w$ can be written as
\begin{equation}\label{eq:word}
w = \prod_{i=1}^n x^{a_i} y^{b_i}
\end{equation}
for some integers $a_i, b_i$ and a positive integer $n$.
Since our objective is to show that the image of $\tilde w$ contains a non-trivial
unipotent element in $\SL_2(K)$, it suffices to show this after replacing $w$ by any of
its conjugates $x w x^{-1}$ or $y^{-1} w y$. We can repeat this procedure of reducing $w$.
If we end with a power of $x$ or $y$,
then the image of the word map clearly contains a non-trivial unipotent.
We will therefore focus on the case when the reduction process
end with a word that is cyclically reduced and of the form
\eqref{eq:word} with $a_i, b_i \neq 0$ for all $1 \leq i \leq n$.
In this situation, we say
that the {\em length} of the word $w$ is $\ell(w) = n$.

% Since our concern is about the image of word maps up to conjugation, we
% say that two words are equivalents if they are conjugate. If two words
% $w_1,w_2$ are equivalent, we will write this as $w_1\sim w_2$. Notice that the
% action of the cyclic permutation on a word gives equivalent words: for
% instance, $xyx^{-1}\sim x^{-1}xy= y$. We put aside the cases $w\sim x^a$ or $w\sim y^b$: they will be treat separately. Therefore, we consider only words written
% as in~\eqref{eq:word} with $a_i,b_i\neq0$ for all $i=1,\ldots,n$. In such a case, we say that $n$
% is the length of the word. 
% %}

\begin{example}
The double commutator word 
$[[x^k, y^l], [x^m, y^n]]$
reduces to
\[
w_{k,l,m,n}
=
x^{m} y^{-l + n} x^{-k} y^{l} x^{k} y^{-n} x^{-m} y^{n} x^{-k + m} y^{-l} x^{k} y^{l} x^{-m} y^{-n}.
\]
As long as $l \neq n$ and $k \neq m$, 
this word is cyclically reduced and of length $7$.
\end{example}

%%%%%%%%%%%%%%%%%%%%%%%%%%%%%%%%%%%%%%%%%%%%%%%%%%%%%%%%%%%%%%%%%%%%%%%%%%%%%%%%%%%%%
\subsection{Wiggling one parameter}
%%%%%%%%%%%%%%%%%%%%%%%%%%%%%%%%%%%%%%%%%%%%%%%%%%%%%%%%%%%%%%%%%%%%%%%%%%%%%%%%%%%%%

Let $R_0 = k[\lambda, \lambda^{-1}, \mu, \mu^{-1}]$ be the commutative ring of Laurent
polynomials on the variables $\lambda, \mu$. Consider a pair of generic semisimple
elements
\[
x = \begin{pmatrix} \lambda & 0 \\ 0 & \lambda^{-1} \end{pmatrix}
\quad \text{and} \quad
y = \begin{pmatrix} \mu & 0 \\ 0 & \mu^{-1} \end{pmatrix}.
\]
For a fixed choice of $\lambda, \mu \in K^*$, we have the {\em wiggle map}
\[
\wiggle^{x,y}_w \colon \SL_2(K) \to \SL_2(K), \quad g \mapsto \tilde w(x, y^g)
\]
that conjugates the second parameter $y$ and returns the $\tilde w$ value.

\begin{lemma} \label{lemma:imageonsemisimpleisunionofvarphis}
\[
\image \tilde w|_{\SL_2(K)_{ss} \times \SL_2(K)_{ss}} 
=
\bigcup_{x,y \in \Torus_2(K)} \bigcup_{h \in \SL_2(K)} \left( \image \wiggle^{x,y}_w \right)^h
\]
\end{lemma}

\begin{proof}
The left hand side is closed under conjugation by elements of $\SL_2(K)$ and it
contains $\image \wiggle^{x,y}_w$. This proves that LHS $\supseteq$ RHS. As for the converse,
consider a pair of elements $x', y' \in \SL_2(K)_{ss}$. Since every semisimple element 
is in a maximal torus and all maximal tori are conjugate, there exist $h, g \in \SL_2(K)$
with $x := (x')^{h^{-1}} \in \Torus_2(K)$ and 
$y := ( (y')^{h^{-1}} )^{g^{-1}} \in \Torus_2(K)$. Then
\[
\tilde w(x', y') = \tilde w \left( x , (y')^{h^{-1}} \right)^{h} = \wiggle^{x,y}_w( g )^h \in \text{RHS}.
\qedhere
\]
\end{proof}

% In this section, we apply the
% properties of the matrices $\xi_g$ to obtain an explicit expression of
% $w(x,y)$ up to conjugation where $x,y\in G$ for $G=\SL_2k$. 
% These results are analogous and extend those that
% compute the trace of the word map used in [REF], [REF] from [Klein]. 
% Our computations will work on the set $G_{ss}$ of semisimple elements of $G=\SL_2k$, but this method can be reformulate to unipotents elements. Let $T$ be the maximal torus {\color{red}(diagonal elements)} of  
% $G=\SL_2k$. Let us define the map $\wiggle^{x,y}_w^{x,y}(g)=\tilde w(x,y^g)$ with $x,y\in T$. 
% We deal with the map $\wiggle^{x,y}_w^{x,y}(g)$, where we treat $x,y$ as parameters, instead $\tilde w$. 
% But there is no loss of generality: given any $x',y'\in G_{ss}$, there
% exists $h\in G$ such that $(x')^h=x\in T$ by the Torus theorem [Lemma~11.12,
% Hall]:
% \begin{equation}\label{eq:relationship_word_phi}
%     w(x',y')=w\left(x,(y')^{h^{-1}}\right)^h=\wiggle^{x,y}_w^{x,y}(g)^h
% \end{equation}
% for some $g\in G$ such that $y=\left((y')^{h^{-1}}\right)^{g^{-1}}\in T$.

Our strategy is to find a non-trivial unipotent as the image of a pair of semisimple
elements under $\tilde w$. It follows from Lemma \ref{lemma:imageonsemisimpleisunionofvarphis}
that it is enough to search for this unipotent in the image of the map $\wiggle^{x,y}_w$ 
for some $x,y \in \Torus_2(K)$. 

\subsection{A normal form of the wiggle}
%%%%%%%%%%%%%%%%%%%%%%%%%%%%%%%%%%%%%%%%%%%%%%%%%%%%%%%%%%%%%%%%%%%%%%%%%%%%%%%%

Write the word $w$ as in \eqref{eq:word}. We can use \eqref{eq:conj} to 
rewrite $\wiggle^{x,y}_w(g)$ in terms of $\xi_g$ as follows: 
%By~\eqref{eq:conj} we have
\begin{equation}\label{eq:phig}
\wiggle^{x,y}_w(g) = 
\prod_{i=1}^n x^{a_i} (y^g)^{b_i} = 
\prod_{i=1}^n x^{a_i} (y^{b_i})^g =
    %=\prod_{i=1}^n \left(x^{a_i}y^{b_i}+x^{a_i}C_g(y^{b_i})\right)=\prod_{i=1}^n
\prod_{i=1}^n \left( x^{a_i} y^{b_i} + x^{a_i} \overline{y^{b_i}} \cdot \xi_g \right).
\end{equation}
We will expand the last product using the following notation.
For a subset $I \subseteq \{1,2,\ldots,n\}$, let
\[
X_I(i)=
\begin{cases}
    x^{a_i} \overline{y^{b_i}} \cdot \xi_g  & \text{ if } i \in I \\
    x^{a_i} y^{b_i}                         & \text{ if } i \not\in I
\end{cases}
\]
and let $W_I = \prod_{i=1}^n X_I(i)$. Then
\begin{equation}\label{eq:phiassum}
\wiggle^{x,y}_w(g) = \sum_{I \subseteq \{1,2,\ldots,n\}} W_I.
\end{equation}
Each summand $W_I$ can be written as
\begin{equation}\label{eq:wIstd}
    W_I = \delta \cdot D_0 \xi_g D_1 \xi_g D_2 \cdots D_{r-1} \xi_g D_r,
\end{equation}
where $r = |I|$, $\delta \in R_0$ and $D_i \in \Diag_2(R_0) \cap \SL_2(R_0)$ for $0 \leq i \leq r$. We now show how
this expression can be further simplified.

\begin{lemma}\label{lemma:wI}
Each $W_I$ can be written as $E_0 + E_1 \xi_g$
with $E_0, E_1 \in \Diag_2(R_0[\det \xi_g])$.
\end{lemma}

\begin{proof}
Write $W_I$ as in~\eqref{eq:wIstd}, where the factor $\delta$ can clearly be ignored.
We use induction on $r$ to prove the claim.
If $r=0$, there is nothing to prove. If $r=1$, we can use 
Lemma~\ref{lemma:xiproperties2} (2) to express $W_I$ in the 
desired form:
\[
    D_0 \xi_g D_1 = \bar D_1 \det \xi_g \cdot D_0 + D_0 D_1^{-1} \xi_g.
\]
Now let $r \geq 2$ be a positive integer and assume by induction
that the expression~\eqref{eq:wIstd} can be written in the desired form 
for all $r' < r$.
Apply Lemma~\ref{lemma:xiproperties2} (2) on the term $\xi_gD_1$ to obtain
\[
D_0 \xi_g D_1 \xi_g D_2 \cdots D_{r-1} \xi_g D_r =
D_0 \left( D_1^{-1} \xi_g^2 D_2 +
\overline{D_1} \det \xi_g \cdot \xi_g D_2 \right)
\xi_g D_3 \cdots \xi_g D_r,
\]
which is the same as 
\[
\det \xi_g \cdot D_0 D_1^{-1} D_2\xi_g\cdots\xi_gD_r +
\overline{D_1} \det \xi_g \cdot D_0\xi_gD_2\cdots\xi_gD_r
\]
by Lemma \ref{lemma:xiproperties}.
By induction, both summands can be written in the desired form.
The same is therefore true for their sum and the proof is complete.
\end{proof}

Let $R = R_0[t]$ be the ring of polynomials with coefficients in $R_0$.
We will identify the ring $\Mat_2(R)$ with $\Mat_2(R_0)[t]$.
In this language, the wiggle can be expressed in the following unique way.

\begin{theorem}\label{theorem:wiggle_normalform}
There exist unique matrix polynomials $A^{x,y}_w, B^{x,y}_w \in \Diag_2(R_0)[t]$ so that
\[
\forall g \in \SL_2(K). \qquad
\wiggle^{x,y}_w(g) = A^{x,y}_w(t) + B^{x,y}_w(t) \cdot \xi_g
\quad \text{at} \quad
t = \det \xi_g.
\]
\end{theorem}

\begin{proof}
Existence follows from \eqref{eq:phiassum} and Lemma \ref{lemma:wI}.
As for uniqueness, suppose one has
$A^{x,y}_w(t) + B^{x,y}_w(t) \xi_g = A'^{x,y}_w(t) + B'^{x,y}_w(t) \xi_g$
at $t = \det \xi_g$ for all $g \in \SL_2(K)$.
It will suffice to consider only those $g \in \SL_2(K)$ with $\det \xi_g \notin \{ 0, -1 \}$.
By Lemma \ref{lemma:xiproperties_formofxi}, we can write
$\xi_g = \left( \begin{smallmatrix} t & p \\ q & - t \end{smallmatrix} \right)$
with $pq = t (t + 1)$.
Then $(B^{x,y}_w(t))^1_1 p = (B'^{x,y}_w(t))^1_1 p$.
This implies that the polynomials $(B^{x,y}_w(t))^1_1$ and $(B'^{x,y}_w(t))^1_1$
have the same values whenever $t \notin \{ 0, -1 \}$, hence they are the same.
A similar argument gives that $(B^{x,y}_w(t))^2_2 = (B'^{x,y}_w(t))^2_2$.
Therefore $B^{x,y}_w = B'^{x,y}_w$ and hence also $A^{x,y}_w = A'^{x,y}_w$.
\end{proof}

By means of Theorem \ref{theorem:wiggle_normalform}, one can compute
$\wiggle^{x,y}_w(g)$ by first computing the two associated matrix polynomials
$A^{x,y}_w(t), B^{x,y}_w(t) \in \Diag_2(R_0)[t]$ and evaluating them
at $t = \det \xi_g$.

%%%%%%%%%%%%%%%%%%%%%%%%%%%%%%%%%%%%%%%%%%%%%%%%%%%%%%%%%%%%%%%%%%%%%%%%%%%%%%%%
\subsection{Computing the associated polynomials}
%%%%%%%%%%%%%%%%%%%%%%%%%%%%%%%%%%%%%%%%%%%%%%%%%%%%%%%%%%%%%%%%%%%%%%%%%%%%%%%%

We now give an effective way of computing the polynomials 
$A^{x,y}_w(t), B^{x,y}_w(t)$. This will rely on a recursive formula
based on shortening the length of the word $w$.

In the case of the trivial word $w = 1 \in \FF_2$, we clearly have
$\wiggle^{x,y}_1(g) = \mathbb{1}$, hence $A^{x,y}_1 = \mathbb{1}$ and $B^{x,y}_1 = 0$.

Suppose now that $\ell(w) \geq 1$. We can thus write $w = w' \cdot x^a y^b$
for some $w' \in \FF_2$ with $\ell(w') < \ell(w)$.
This induces 
\[
\wiggle^{x,y}_w(g) = \wiggle_{w'}^{x,y}(g) \cdot \wiggle_{x^a y^b}^{x,y}(g),
\]
which is the same as
\[
A^{x,y}_w(t) + B^{x,y}_w(t) \xi_g = 
\left( A^{x,y}_{w'}(t) + B^{x,y}_{w'}(t) \xi_g \right)
\left(  x^a y^b + \overline{y^b} \cdot x^a \xi_g \right)
\]
at $t = \det \xi_g$ by Theorem \ref{theorem:wiggle_normalform}.
Expand the product on the right side and use Lemma~\ref{lemma:xiproperties2} (2) to
write the obtained in a normal form as in 
Theorem \ref{theorem:wiggle_normalform}.
% $\xi_gx^ay^b=x^{-a}y^{-b}\xi_g+\overline{x^ay^b}\cdot t\mathbb{1}$.
Uniqueness of normal forms then gives the following recursive formulae:
\[
\left\{
\begin{array}{l}
A^{x,y}_w(t) = 
A^{x,y}_{w'}(t) x^ay^b 
+ B^{x,y}_{w'}(t) \overline{x^ay^b} t 
- B^{x,y}_{w'}(t) \overline{y^b} x^{-a} t,\\[0.5em]

B^{x,y}_w(t) = 
A^{x,y}_{w'}(t) \overline{y^b} x^a 
+ B^{x,y}_{w'}(t) x^{-a} y^{-b} 
+ B^{x,y}_{w'}(t) \overline{x^a} \cdot \overline{y^b} t.
\end{array}
\right.
\]
The first equality can be simplified using the fact that
$\overline{x^a y^b} \cdot \mathbb{1} = \overline{x^a} y^b + \overline{y^b} x^{-a}$.
We then arrive to the following recursion:
\[
\left\{
\begin{array}{l}
A^{x,y}_w(t) = 
A^{x,y}_{w'}(t) \cdot x^ay^b 
+ B^{x,y}_{w'}(t) \cdot \overline{x^a} y^b t,\\[0.5em]

B^{x,y}_w(t) = 
A^{x,y}_{w'}(t) \cdot \overline{y^b} x^a 
+ B^{x,y}_{w'}(t) \cdot 
\left( x^{-a} y^{-b} + \overline{x^a} \cdot \overline{y^b} t \right).
\end{array}
\right.
\]
The formulae can be thought of as linear in the recursive variables
$A^{x,y}_{w'}(t), B^{x,y}_{w'}(t)$ over the ring $\Diag_2(R)$. We can
therefore express them in terms of matrices. Note that the two particular
cases $(a,b) \in \{ (1,0), (0,1) \}$ of the recursion are
\begin{align*}
A^{x,y}_{w \cdot x}(t) &= A^{x,y}_w(t) \cdot x +  B^{x,y}_w(t) \cdot \overline{x}t,&
B^{x,y}_{w \cdot x}(t) &= B^{x,y}_w(t) \cdot x^{-1},\\
A^{x,y}_{w \cdot y}(t) &= A^{x,y}_w(t) \cdot y, & 
B^{x,y}_{w \cdot y}(t) &= A^{x,y}_{w}(t) \cdot \overline{y} + B^{x,y}_w(t) \cdot y^{-1}.
\end{align*}
In order to write these in matrix form, let
\[
X = \begin{pmatrix} x & \overline{x} t \cdot \mathbb{1} \\ 0 & x^{-1} \end{pmatrix}
\quad \text{and} \quad
Y = \begin{pmatrix} y & 0 \\ \overline{y} \cdot \mathbb{1} & y^{-1} \end{pmatrix}
\]
be matrices in $\SL_2(\Diag_2(R))$. We then have
\[
\begin{pmatrix} 
A^{x,y}_w(t) \\[0.5em] B^{x,y}_w(t)
\end{pmatrix} 
=
Y^b X^a
\begin{pmatrix}
A^{x,y}_{w'}(t) \\[0.5em] B^{x,y}_{w'}(t) 
\end{pmatrix}.
\]
Extending this step by step to the whole word $w$ gives the following way of
quickly computing the associated polynomials of the wiggle map.

\begin{proposition} \label{prop:calculating_associated_polynomials}
Let $w = \prod_{i=1}^n x^{a_i} y^{b_i}$ and $\overleftarrow{w} = \prod_{i=n}^1 y^{b_i} x^{a_i}$.
Then
\[
\begin{pmatrix} 
A^{x,y}_w(t) \\[0.5em] B^{x,y}_w(t)
\end{pmatrix} 
=
\overleftarrow{w}(X,Y) \cdot 
\begin{pmatrix}
\mathbb{1} \\[0.5em] 0 
\end{pmatrix}.
\]
\end{proposition}

We record an immediate corollary.

\begin{corollary}
The polynomials $A^{x,y}_w(t), B^{x,y}_w(t) \in \Diag_2(R_0)[t]$ are coprime. 
\end{corollary}

\begin{proof}
It follows from Proposition \ref{prop:calculating_associated_polynomials} that the first 
column of the matrix $\overleftarrow{w}(X,Y)$ is
$\left( A^{x,y}_w(t), B^{x,y}_w(t) \right)^T$.
It now follows from $\det \overleftarrow{w}(X,Y) = \mathbb{1}$ that
\[
A^{x,y}_w(t) \cdot \overleftarrow{w}(X,Y)^2_2
-
B^{x,y}_w(t) \cdot \overleftarrow{w}(X,Y)^1_2
=
\mathbb{1}. \qedhere
\]
\end{proof}

%%%%%%%%%%%%%%%%%%%%%%%%%%%%%%%%%%%%%%%%%%%%%%%%%%%%%%%%%%%%%%%%%%%%%%%%%%%%%%%%
\subsection{Symmetry and the shape of the wiggle}
%%%%%%%%%%%%%%%%%%%%%%%%%%%%%%%%%%%%%%%%%%%%%%%%%%%%%%%%%%%%%%%%%%%%%%%%%%%%%%%%%

Let
$\sigma = \left( \begin{smallmatrix} 0 & 1 \\ 1 & 0 \end{smallmatrix} \right) \in \SL_2(K)$.
We will exploit the action of $\sigma$ on $\SL_2(K)$ by conjugation to obtain a
certain symmetrical relation between the polynomials defining the associated matrix
polynomials. Note first that we have $x^\sigma = x^{-1}$ and $y^\sigma = y^{-1}$. 
Now
\[
\wiggle^{x,y}_w(g^\sigma)^\sigma = 
w(x,y^{g^\sigma})^\sigma = 
w(x^\sigma, y^{g^\sigma \sigma}) = 
w(x^\sigma, (y^\sigma)^{g}) =
\wiggle^{x^{-1}, y^{-1}}_w(g).
\]
It follows from Theorem \ref{theorem:wiggle_normalform} that we can therefore
write
\[
\wiggle^{x^{-1}, y^{-1}}_w(g) =
\left( 
A^{x,y}_w(\det \xi_{g^\sigma}) + 
B^{x,y}_w(\det \xi_{g^\sigma}) \cdot \xi_{g^\sigma} 
\right)^\sigma.
\]
By Lemma \ref{lemma:xiproperties} (3), we have $\xi_{g^\sigma} = - \xi_g^\sigma$. 
This gives
\[
A^{x^{-1},y^{-1}}_w(t) +
B^{x^{-1},y^{-1}}_w(t) \cdot \xi_{g} 
=
\wiggle^{x^{-1}, y^{-1}}_w(g) =
A^{x,y}_w(t)^\sigma -
B^{x,y}_w(t)^\sigma \cdot \xi_{g} 
\]
at $t = \det \xi_g$.
Using the uniqueness part of Theorem \ref{theorem:wiggle_normalform},
we now conclude the following.

\begin{proposition}[Symmetry of associated polynomials]
\label{prop:symmetry_of_associated_polynomials}
\[
A^{x, y}_w(t)^\sigma = A^{x^{-1}, y^{-1}}_w(t),
\quad
B^{x, y}_w(t)^\sigma = - B^{x^{-1}, y^{-1}}_w(t).
\]
\end{proposition}

This symmetry allows us to reduce associated matrix polynomials to ordinary
polynomials. We will express this in terms of an extended action of $\sigma$.
First of all, $\sigma$ acts by conjugation on $\Mat_2(R_0)$. By letting $t^\sigma = t$,
we have an induced action of $\sigma$ on $\Mat_2(R)$. This action induces the
standard action on matrix functions. For example, for a polynomial with matrix 
parameters $\delta^{x,y}(t) \in R_0[t]$, we have 
$\delta^{x,y}(t)^\sigma = \delta^{x^\sigma, y^\sigma}(t)$.

\begin{corollary} \label{cor:matrix_polynomials_to_ordinary}
There exist polynomials $\alpha^{x,y}_w, \beta^{x,y}_w \in R_0[t]$ so that
\[
A^{x,y}_w(t) = 
\begin{pmatrix}
\alpha^{x,y}_w(t) & 0 \\[0.5em] 0 & \alpha^{x,y}_w(t)^\sigma
\end{pmatrix},
\quad
B^{x,y}_w(t) = 
\begin{pmatrix}
\beta^{x,y}_w(t) & 0 \\[0.5em] 0 & - \beta^{x,y}_w(t)^\sigma
\end{pmatrix}.
\]
\end{corollary}

\begin{example}
Using Proposition \ref{prop:calculating_associated_polynomials}
and extracting only the upper-left entries, we can compute the
polynomials $\alpha$ and $\beta$ associated to the double commutator word
$w_{k,l,m,n}$.
Both are, in the general situation, of degree $6$.
We have 
\begin{align*}
\alpha^{x,y}_{w_{k,l,m,n}}(t) &= 
\frac{
( \lambda ^{2 k}-1 ) 
( \mu ^{2 l}-1 ) 
( \lambda ^{2 m}-1 ) 
( \mu ^{2 n}-1 ) 
}{ \lambda ^{4 (k+m)} \mu ^{4 (l+n)} }
\cdot
t ( t + 1 )
\cdot
\alpha(t) + 1, \\
\beta^{x,y}_{w_{k,l,m,n}}(t) &= 
\frac{
( \lambda ^{2 k}-1 ) 
( \mu ^{2 l}-1 ) 
( \lambda ^{2 m}-1 ) 
( \mu ^{2 n}-1 ) 
}{ \lambda ^{4 (k+m)} \mu ^{4 (l+n)} }
\cdot
t ( t + 1 )
\cdot
\beta(t),
\end{align*}
where 
$\alpha(t) = \sum_{i = 0}^4 a_i t^i$ 
and
$\beta(t) = \sum_{i = 0}^4 b_i t^i$.
The coefficients of $\alpha$ are
%and
%$\beta^{x,y}_{w_{k,l,m,n}}(t) = \sum_{i = 0}^6 b_i t^i$.
\begin{align*}

\begin{autobreak}
a_0 =
\lambda ^{2 (k+2 m)} \mu ^{2 (l+2 n)}-\lambda ^{2 (2 k+m)} \mu ^{2 (2 l+n)},
\end{autobreak}
\\
\begin{autobreak}
a_1 =
-\lambda ^{2 m} 
( \mu ^{2 (l+n)} 
( \mu ^{2 l}-1 ) 
( \mu ^{2 n}-1 ) \lambda ^{2 k}-\mu ^{2 l} 
( \mu ^{2 l}-1 ) 
( \mu ^{2 l}-3 \mu ^{2 n}
+\mu ^{4 n} ) \lambda ^{4 k}+\mu ^{2 l} 
( \mu ^{2 l}-1 ) 
( \mu ^{2 l}-\mu ^{2 n} ) \lambda ^{6 k}-\mu ^{4 n} 
( \mu ^{2 l}-1 )^2 \lambda ^{2 m}-\mu ^{4 n} 
( \mu ^{2 l}-1 ) \lambda ^{4 m}+
( \mu ^{4 n}+\mu ^{4 (l+n)}
+\mu ^{2 (2 l+n)}-3 \mu ^{2 (l+2 n)} ) 
\lambda^{2 (k+m)}+\mu ^{2 (l+n)} 
( -2 \mu ^{2 l}+\mu ^{2 n}+1 ) \lambda ^{2 (2 k+m)}+\mu ^{4 n} 
( \mu ^{2 l}-1 ) \lambda ^{2 (k+2 m)} ),
\end{autobreak}
\\
\begin{autobreak}
a_2 =
\lambda ^{2 m} 
( \lambda ^{2 k}-1 ) 
( \mu ^{2 l}-1 ) 
( -
( \mu ^{2 n}-1 ) 
( \mu ^{4 l}+\mu ^{2 n}
-3 \mu ^{2 (l+n)} ) \lambda ^{2 k}+
( \mu ^{2 l}-3 \mu ^{4 l}
-\mu ^{2 n}+3 \mu ^{2 (l+n)}
+\mu ^{2 (2 l+n)}-\mu ^{2 (l+2 n)} ) \lambda ^{4 k}-\mu ^{2 n} 
( \mu ^{2 l}-1 ) 
( -\mu ^{2 l}+3 \mu ^{2 n}-1 ) \lambda ^{2 m}+\mu ^{2 n} 
( \mu ^{2 l}-3 \mu ^{2 n}
+\mu ^{2 (l+n)}+1 ) \lambda ^{4 m}-\mu ^{2 l} 
( \mu ^{2 l}-1 ) 
( \mu ^{2 n}-1 ) \lambda ^{2 (k+m)}+
( \mu ^{2 l}-1 ) 
( \mu ^{2 l}-\mu ^{2 n} ) \lambda ^{2 (2 k+m)}-
( \mu ^{2 l}-\mu ^{2 n} ) 
( \mu ^{2 n}-1 ) \lambda ^{2 (k+2 m)} ),
\end{autobreak}
\\
\begin{autobreak}
a_3 =
-\lambda ^{2 m} 
( \lambda ^{2 k}-1 ) 
( \mu ^{2 l}-1 ) 
( -
( \mu ^{2 n}-1 ) 
( \mu ^{2 l}-2 \mu ^{4 l}
-2 \mu ^{2 n}+3 \mu ^{2 (l+n)} ) \lambda ^{2 k}-
( 2 \mu ^{2 l}-3 \mu ^{4 l}
-2 \mu ^{2 n}+\mu ^{4 n}
+2 \mu ^{2 (l+n)}+2 \mu ^{2 (2 l+n)}
-2 \mu ^{2 (l+2 n)} ) \lambda ^{4 k}+
( -\mu ^{2 l}+\mu ^{4 l}
+2 \mu ^{2 n}-3 \mu ^{4 n}
-2 \mu ^{2 (2 l+n)}+3 \mu ^{2 (l+2 n)} ) \lambda ^{2 m}+
( \mu ^{2 l}-2 \mu ^{2 n}
+3 \mu ^{4 n}-\mu ^{2 (l+n)}
+\mu ^{2 (2 l+n)}-2 \mu ^{2 (l+2 n)} ) \lambda ^{4 m}+\mu ^{2 l} 
( \mu ^{2 l}-1 ) 
( \mu ^{2 n}-1 ) \lambda ^{2 (k+m)}+
( \mu ^{2 l}-1 ) 
( \mu ^{2 l}-\mu ^{2 n} ) 
( \mu ^{2 n}-2 ) \lambda ^{2 (2 k+m)}+
( \mu ^{2 l}-2 ) 
( \mu ^{2 n}-1 ) 
( \mu ^{2 n}-\mu ^{2 l} ) \lambda ^{2 (k+2 m)} ), 
\end{autobreak}
\\
\begin{autobreak}
a_4 =
\lambda ^{2 m} 
( \lambda ^{2 k}-1 )^2 
( \lambda ^{2 m}-1 ) 
( \lambda ^{2 m}-\lambda ^{2 k} ) 
( \mu ^{2 l}-1 )^2 
( \mu ^{2 l}-\mu ^{2 n} ) 
( \mu ^{2 n}-1 ),
\end{autobreak}
\end{align*}
and expressions of the same form can be computed for $\beta$.
We remark that 
$\alpha(0), \alpha(-1)$ and
$\beta(0), \beta(-1)$
are all non-zero
for general $k,l,m,n$
and generic $\lambda, \mu$.
\end{example}

Using the symmetry under $\sigma$, we can now express the wiggle matrix as follows.

\begin{corollary} \label{cor:shape_of_wiggle}
Let $\gamma^{x,y}_w(t) = \alpha^{x,y}_w(t) + \beta^{x,y}_w(t) \cdot t \in R_0[t]$.
Then
\[
\wiggle^{x,y}_w(g) =
\begin{pmatrix} 
    \gamma^{x,y}_w(t) & \beta^{x,y}_w(t) \cdot p \\[0.5em]
    - \beta^{x,y}_w(t)^\sigma \cdot q & \gamma^{x,y}_w(t)^\sigma
\end{pmatrix}
\quad \text{at $t = \det \xi_g$},
\]
where 
$\xi_g = \left( \begin{smallmatrix} t & p \\ q & - t \end{smallmatrix} \right)$
as in Lemma \ref{lemma:xiproperties_formofxi}.
\end{corollary}

\begin{proof}
This is immediate by joining 
Theorem \ref{theorem:wiggle_normalform} and
Corollary \ref{cor:matrix_polynomials_to_ordinary}.
\end{proof}

\begin{example} \label{ex:computation_gamma_and_trace_for_double_commutator}
For the double commutator word $w_{k,l,m,n}$, we compute
\begin{align*}
\gamma^{x,y}_{w_{k,l,m,n}}(t) &=
\frac{
( \lambda ^{2 k}-1 ) 
( \mu ^{2 l}-1 ) 
( \lambda ^{2 m}-1 ) 
( \mu ^{2 n}-1 ) 
}{ \lambda ^{4 (k+m)} \mu ^{4 (l+n)} }
\cdot
t ( t + 1 )
\cdot
\gamma(t) + 1, \\
\tr \wiggle^{x,y}_{w_{k,l,m,n}}(g) &= 
\left(
\frac{
( \lambda ^{2 k}-1 )
( \mu ^{2 l}-1 )
( \lambda ^{2 m}-1 )
( \mu ^{2 n}-1 )
}{ \lambda ^{2 (k+m)} \mu ^{2 (l+n)} }
\right)^2
\cdot
t^2 ( t + 1 )^2
\cdot
\tau(t) + 2,
\end{align*}
where 
$\gamma(t) = \sum_{i = 0}^5 c_i t^i$
and
$\tau(t) = \sum_{i = 0}^3 t_i t^i$. 
The coefficients of $\tau$ are
\begin{align*}

\begin{autobreak}
t_0 =
( \lambda ^{2 k} \mu ^{2 l}-\lambda ^{2 m} \mu ^{2 n} )^2,
\end{autobreak}
\\
\begin{autobreak}
t_1 =
\mu ^{2 l} 
( \mu ^{2 l}-\mu ^{2 n} ) 
( \mu ^{2 n}-1 ) \lambda ^{2 k}-\mu ^{2 l} 
( -3 \mu ^{2 l}+\mu ^{2 n}+
\mu ^{2 (l+n)}+1 ) \lambda ^{4 k}+\mu ^{2 n} 
( \mu ^{2 l}-1 ) 
( \mu ^{2 n}-\mu ^{2 l} ) \lambda ^{2 m}-\mu ^{2 n} 
( \mu ^{2 l}-3 \mu ^{2 n}+\mu ^{2 (l+n)}+1 ) \lambda ^{4 m}+
( \mu ^{2 l}-1 ) 
( \mu ^{2 n}-1 ) 
( \mu ^{2 l}+\mu ^{2 n} ) \lambda ^{2 (k+m)}-
( \mu ^{2 l}-1 ) 
( \mu ^{2 l}-\mu ^{2 n} ) \lambda ^{2 (2 k+m)}+
( \mu ^{2 l}-\mu ^{2 n} ) 
( \mu ^{2 n}-1 ) \lambda ^{2 (k+2 m)},
\end{autobreak}
\\
\begin{autobreak}
t_2 =
( 2 \mu ^{2 l}-1 ) 
( \mu ^{2 l}-\mu ^{2 n} ) 
( \mu ^{2 n}-1 ) \lambda ^{2 k}+
( -2 \mu ^{2 l}+3 \mu ^{4 l}
+\mu ^{2 n}-\mu ^{2 (l+n)}
-2 \mu ^{2 (2 l+n)}+\mu ^{2 (l+2 n)} ) \lambda ^{4 k}+
( \mu ^{2 l}-1 ) 
( \mu ^{2 n}-\mu ^{2 l} ) 
( 2 \mu ^{2 n}-1 ) \lambda ^{2 m}+
( \mu ^{2 l}-2 \mu ^{2 n}
+3 \mu ^{4 n}-\mu ^{2 (l+n)}
+\mu ^{2 (2 l+n)}-2 \mu ^{2 (l+2 n)} ) \lambda ^{4 m}+
( \mu ^{2 l}-1 ) 
( \mu ^{2 n}-1 ) 
( \mu ^{2 l}+\mu ^{2 n} ) \lambda ^{2 (k+m)}+
( \mu ^{2 l}-1 ) 
( \mu ^{2 l}-\mu ^{2 n} ) 
( \mu ^{2 n}-2 ) \lambda ^{2 (2 k+m)}+
( \mu ^{2 l}-2 ) 
( \mu ^{2 n}-1 ) 
( \mu ^{2 n}-\mu ^{2 l} ) \lambda ^{2 (k+2 m)},
\end{autobreak}
\\
\begin{autobreak}
t_3 =
( \lambda ^{2 k}-1 ) 
( \lambda ^{2 k}-\lambda ^{2 m} ) 
( \lambda ^{2 m}-1 ) 
( \mu ^{2 l}-1 ) 
( \mu ^{2 l}-\mu ^{2 n} ) 
( \mu ^{2 n}-1 ).
\end{autobreak}
\end{align*}
\end{example}

In the next section, we will make use of the following
relationship between the polynomials $\gamma^{x,y}_w$ and $\beta^{x,y}_w$.

\begin{corollary} \label{cor:determinant_of_wiggle}
$
\gamma^{x,y}_w(t) \cdot \gamma^{x,y}_w(t)^\sigma -
\beta^{x,y}_w(t) \cdot \beta^{x,y}_w(t)^\sigma \cdot t(t+1) = 1.
$
\end{corollary}

\begin{proof}
Use $\det \wiggle^{x,y}_w(g) = 1$ together with
Corollary \ref{cor:shape_of_wiggle} and 
Lemma \ref{lemma:xiproperties_formofxi}.
\end{proof}

%%%%%%%%%%%%%%%%%%%%%%%%%%%%%%%%%%%%%%%%%%%%%%%%%%%%%%%%%%%%%%%%%%%%%%%%%%%%%%%%
%%%%%%%%%%%%%%%%%%%%%%%%%%%%%%%%%%%%%%%%%%%%%%%%%%%%%%%%%%%%%%%%%%%%%%%%%%%%%%%%
\section{Finding non-trivial unipotents}
\label{sec:finding_unipotents}
%%%%%%%%%%%%%%%%%%%%%%%%%%%%%%%%%%%%%%%%%%%%%%%%%%%%%%%%%%%%%%%%%%%%%%%%%%%%%%%%
%%%%%%%%%%%%%%%%%%%%%%%%%%%%%%%%%%%%%%%%%%%%%%%%%%%%%%%%%%%%%%%%%%%%%%%%%%%%%%%%

%%%%%%%%%%%%%%%%%%%%%%%%%%%%%%%%%%%%%%%%%%%%%%%%%%%%%%%%%%%%%%%%%%%%%%%%%%%%%%%%
\subsection{Reducing to one parameter}
%%%%%%%%%%%%%%%%%%%%%%%%%%%%%%%%%%%%%%%%%%%%%%%%%%%%%%%%%%%%%%%%%%%%%%%%%%%%%%%%%

Set
\[
\overline{\unipotents^{x,y}_w} = 
\{ g \in \SL_2(K)
\mid 
\wiggle^{x,y}_w(g) \text{ is a non-trivial unipotent}
\}.
\]
Our objective is to show that this is a non-empty set.
Since the wiggle can be expressed almost exclusively in terms of 
the parameter $t = \det \xi_g$ by Corollary \ref{cor:shape_of_wiggle},
we will rather work with a slight
modification of the set $\overline{\unipotents^{x,y}_w}$ defined
as follows:
\[
\unipotents^{x,y}_w = 
\left\{ t \in K - \{ 0, -1 \}
\;\middle|\;
\begin{array}{l}
    \gamma^{x,y}_w(t) + \gamma^{x,y}_w(t)^\sigma = 2 \text{ and } \\[1ex]
    \left( 
\beta^{x,y}_w(t) \neq 0
\text{ or }
\beta^{x,y}_w(t)^\sigma \neq 0
\right)
\end{array}
\right\}.
\]
We now explain the connection between these two sets in terms of
the map
\[
\xi \colon \SL_2(K) \to \Mat_2(K),
\quad
g \mapsto \xi_g.
\]

\begin{lemma}
$
(\det \! \circ \xi)^{-1} \left( \unipotents^{x,y}_w \right)
\subseteq
\overline{\unipotents^{x,y}_w}.
$
\end{lemma}

\begin{proof}
Recall first that an element $v \in \SL_2(K)$ is a non-trivial unipotent if and
only if $\tr v = 2$ and $v \neq \mathbb{1}$.
Given $t \in \unipotents^{x,y}_w$, let $g \in \SL_2(K)$ be any element with
$\det \xi_g = t$. Note that after writing
$\xi_g = \left( \begin{smallmatrix} t & p \\ q & - t \end{smallmatrix} \right)$
as in Lemma \ref{lemma:xiproperties_formofxi},
we have $pq = - t(t+1) \neq 0$.
It then follows from Corollary \ref{cor:shape_of_wiggle} 
that $\tr \wiggle^{x,y}_w(g) = \gamma^{x,y}_w(t) + \gamma^{x,y}_w(t)^\sigma = 2$
and at least one of $\wiggle^{x,y}_w(g)^1_2$, $\wiggle^{x,y}_w(g)^2_1$ is non-zero.
Therefore $\wiggle^{x,y}_w(g)$ is a non-trivial unipotent in $\SL_2(K)$,
and so $g \in \overline{\unipotents^{x,y}_w}$.
\end{proof}

It follows from Remark \ref{remark:xig_explicit} that the composition map 
$\det \! \circ \xi \colon \SL_2(K) \to k$ is surjective. In order to find
a non-trivial unipotent, it will therefore suffice to show that the set
$\unipotents^{x,y}_w$ is non-empty for some $x,y \in \Torus_2(K)$.

\subsection{Reducing to one polynomial}
%%%%%%%%%%%%%%%%%%%%%%%%%%%%%%%%%%%%%%%%%%%%%%%%%%%%%%%%%%%%%%%%%%%%%%%%%%%%%%%%%

We will in fact show that the following subset of $\unipotents^{x,y}_w$ is 
non-empty.

\begin{lemma} \label{lemma:unipotents_only_in_terms_of_gamma}
$
\{
t \in K -  \{ 0, -1 \} \mid
\gamma^{x,y}_w(t) + \gamma^{x,y}_w(t)^\sigma = 2
\text{ and }
\gamma^{x,y}_w(t) \neq 1
\}
\subseteq
\unipotents^{x,y}_w.
$
\end{lemma}

\begin{proof}
Let $t \in$ LHS.
It follows from Corollary \ref{cor:determinant_of_wiggle} that
\[
0 = 
\gamma^{x,y}_w(t) \cdot 
\left(
\gamma^{x,y}_w(t) + \gamma^{x,y}_w(t)^\sigma - 2
\right) =
\left( \gamma^{x,y}_w(t) - 1 \right)^2 + \beta^{x,y}_w(t) \cdot \beta^{x,y}_w(t)^\sigma \cdot t(t+1).
\]
Hence $\beta^{x,y}_w(t) \cdot \beta^{x,y}_w(t)^\sigma \neq 0$, which implies that $t \in \unipotents^{x,y}_w$.
\end{proof}

%%%%%%%%%%%%%%%%%%%%%%%%%%%%%%%%%%%%%%%%%%%%%%%%%%%%%%%%%%%%%%%%%%%%%%%%%%%%%%%%
\subsection{Finding suitable roots}
%%%%%%%%%%%%%%%%%%%%%%%%%%%%%%%%%%%%%%%%%%%%%%%%%%%%%%%%%%%%%%%%%%%%%%%%%%%%%%%%

We now show that in the case of the double commutator word,
the subset appearing in Lemma \ref{lemma:unipotents_only_in_terms_of_gamma}
is indeed non-empty (apart from an exceptional case that we translate
to a non-exceptional case). 

\begin{proof}[Proof of Theorem \ref{theorem:double_commutators_are_surjective}]
It now suffices to prove that there exists a root
of the polynomial 
$\gamma^{x,y}_{w_{k,l,m,n}}(t) + \gamma^{x,y}_{w_{k,l,m,n}}(t)^\sigma - 2$ 
that is distinct from $0, -1$ 
and is not a root of $\gamma^{x,y}_{w_{k,l,m,n}}(t)$.
To this end, we use the computations 
from Example \ref{ex:computation_gamma_and_trace_for_double_commutator}.
It then suffices to show that $\tau(t)$ and $\gamma(t)$ have no common roots.
We verify this by computing their resultant:
\begin{align*}
\begin{autobreak}
\res(\tau(t), \gamma(t)) =
-\lambda ^{6 k+8 m}
\mu ^{8 l+6 n} 
( \lambda ^{2k}-1 )^3 
( \mu ^{2l}-1 )^4 
( \lambda ^{2m}-1 ) 
( \mu ^{2n}-1 )^6 
( \lambda ^{2k}-\lambda ^{2m} )^6 
( \mu ^{2l}-\mu ^{2n} )^3 
( \mu ^{2l} \lambda ^{2m}-\lambda ^{2k} \mu ^{2n} ) 
( \lambda ^{2m} \mu ^{2n}-\lambda ^{2k} \mu ^{2l} ) 
( (\mu^{2l} - \mu^{2n}) (\lambda^{2k + 2m} - 1 )
+ (\lambda^{2m} - \lambda^{2k}) (\mu^{2l + 2n} - 1 ) )^2.
\end{autobreak}
\end{align*}
Apart from exceptional cases for the values of $k,l,m,n$, this resultant
is non-zero for generic $\lambda, \mu$, thus proving our claim.
The exceptional cases are the cases when the double commutator word 
$[[x^k, y^l], [x^m, y^n]]$
is not trivial, yet the resultant is zero. By inspecting the factors of
the resultant, we see that the exceptional cases are:
\begin{enumerate}
\item $m = k$ and $n \neq l$.
In this case, the resultant vanishes because the degrees of the
polynomials $\tau$ and $\gamma$ both drop by $1$.
More precisely, in this case we have
\[
\tau(t) = 
-\lambda ^{2 k} ( \mu ^{2l}-\mu ^{2n} )^2
\left(t (\lambda ^{2 k}-1) - 1\right) 
\left(t (\lambda ^{2 k}- 1) + \lambda ^{2 k} \right),
\]
and evaluating $\gamma$ at the two roots of $\tau$ gives the 
generically non-zero values
\[
\left\{\frac{\lambda ^{6 k} \left(\lambda ^{2 k}+1\right) \mu ^{2 (2 l+n)} \left(\mu ^{2 n}-\mu ^{2 l}\right)}{\lambda ^{2 k}-1},-\frac{\lambda ^{6 k} \left(\lambda ^{2
   k}+1\right) \mu ^{2 n} \left(\mu ^{2 n}-\mu ^{2 l}\right)}{\lambda ^{2 k}-1}\right\}.
\]

\item $n = l$ and $m \neq k$.
In this case, the extraordinary event that $\tau$
divides $\gamma$ occurs.
We deal with this case in the following manner.
Using the fact that $[x, y^{-1}] = [y,x]^{x^{-1}}$,
we deduce that the double commutator words 
$[[x^k, y^l], [x^m, y^n]]$
and
$[[y^l, x^k], [y^n, x^m]]$
are conjugate in $\FF_2$.
Hence this case follows from having already proved
that $\image w_{k,l,k,n}$ contains a non-trivial unipotent.

\item $m = - k$ and $n = - l$.
In this case, the polynomial $\tau$ is equal to
\[
\frac{
(\lambda ^{4 k}-1) (\mu ^{4 l}-1)
}{
\lambda ^{4 k} \mu ^{4 l} 
}
\left(
t (\lambda ^{2 k}-1) (\mu ^{2 l}-1)
+ 
\lambda ^{2 k} \mu ^{2 l}+1
\right)^2
\left(
t + 
\frac{
(\lambda ^{2 k} \mu ^{2 l}-1)^2
}{
(\lambda ^{4 k}-1) (\mu ^{4 l}-1)
}
\right).
\]
The double root of $\tau$ is also a root of $\gamma$.
Evaluating $\gamma$ at the third root of $\tau$ gives, however, the generically non-zero value
\[
-\frac{\left(\lambda ^{2 k}+1\right) \left(\lambda ^{2k}-\mu ^{2l}\right) \left(\lambda ^{2k} \mu ^{2l}-1\right) \left(\lambda ^{2 k}+4 \lambda ^{2 k} \mu ^{2 l}+\lambda ^{2 k} \mu ^{4 l}+\lambda ^{4 k} \mu ^{2 l}+\mu ^{2 l}\right)}{
\lambda ^{4 k} 
   \left(\lambda ^{2k} - 1\right) \left(\mu ^{2 l}+1\right)^4}.
\]

\end{enumerate}
\end{proof}

\bibliography{refs}
\bibliographystyle{alpha}

\end{document}